\documentclass[11pt]{amsart}

\usepackage[margin=1.1in]{geometry}
\usepackage[T1]{fontenc}
\usepackage{float}
\usepackage{lmodern}
\usepackage{amsmath,amssymb,amsthm,mathtools}
\usepackage{enumitem}
\usepackage[colorlinks=true,linkcolor=blue,citecolor=blue,urlcolor=blue]{hyperref}
\usepackage[nameinlink,capitalise]{cleveref}
\crefname{conjecture}{Conjecture}{Conjectures}
\usepackage{mathtools}
\usepackage{microtype}

\usepackage{soul}

\newtheorem{theorem}{Theorem}
\newtheorem{conjecture}{Conjecture}

\newtheorem{lemma}{Lemma}

\theoremstyle{definition}
\newtheorem*{remark}{Remark}

\DeclareMathOperator{\spol}{sp}
\DeclareMathOperator{\sr}{sr}
\DeclareMathOperator{\num}{num}
\DeclareMathOperator{\den}{den}

\newcommand{\floor}[1]{\left\lfloor #1\right\rfloor}
\newcommand{\bbfm}{Ballantine--Beck--Feigon--Maurischat}
\newcommand{\ii}{\mathrm{i}}
\newcommand{\OO}{\mathcal O}

\newcommand{\TT}{\mathcal T}
\newcommand{\BB}{\mathcal B}
\newcommand{\ZZ}{\mathbb Z}

\title{Reciprocals of Partition Polynomials}
\author{Evan Chen, Ken Ono and Jujian Zhang}

\address{Axiom Math, 124 University Avenue, Palo Alto, CA 94301}

\email{evan@axiommath.ai}
\email{ken@axiommath.ai}
\email{jujian@axiommath.ai}
\date{May 19, 2026}

\begin{document}

\begin{abstract}
Ballantine--Beck--Feigon--Maurischat introduced the subsum polynomial
\[
\operatorname{sp}(\lambda,x):=\prod_i (1+x^{\lambda_i})
\]
attached to an integer partition $\lambda$, and studied rational functions obtained by
summing reciprocals of these polynomials over natural classes of partitions. They posed ten
conjectures which naturally divide into coprimality and divisibility questions, special-value and
recurrence formulas, and coefficient-shape problems.
We prove all of the
conjectures in the first two families: the ordinary and binary coprimality/divisibility
conjectures, and the odd and ternary special-value/recurrence conjectures.
AxiomProver autonomously produced Lean/mathlib formalizations and
machine-checkable proofs of these six conjectures, and 
also discovered a counterexample to the statement as printed; the corrected form remains open.
\end{abstract}

\maketitle

\section{Introduction}
\subsection{Background}
The study of integer partitions is a classical subject in number theory and
combinatorics; we follow the standard notation and terminology described, for
example, in Andrews' book \cite{Andrews}.  A \emph{partition} of a nonnegative
integer $n$ is a finite nonincreasing sequence
\[
  \lambda=(\lambda_1,\lambda_2,\ldots,\lambda_k)
\]
of positive integers whose sum is $n$.  We write $\lambda\vdash n$.  The
empty sequence is the unique partition of $0$.  The length of $\lambda$ is
$\ell(\lambda) \coloneq k$, and $m_\lambda(i)$ denotes the multiplicity of the part
$i$ in $\lambda$.  In particular, we have
\[
  \sum_{i\ge1} i\,m_\lambda(i)=n
  \qquad {\text {\rm and}}\qquad
  \ell(\lambda)=\sum_{i\ge1}m_\lambda(i).
\]
\bbfm\ recently introduced \cite{BBFM} a family of interesting rational functions attached
to partitions.  We now recall their construction in the notation used
throughout this paper.  For a partition $\lambda$, its \emph{subsum polynomial}
is
\[
  \spol(\lambda,x) \coloneq \prod_{j=1}^k(1+x^{\lambda_j}).
\]
Equivalently, we have
\[
  \spol(\lambda,x)=\prod_{i\ge1}(1+x^i)^{m_\lambda(i)}.
\]
This polynomial is the generating function for submultisets of the multiset of
parts of $\lambda$.  The reciprocal sum over ordinary partitions is
\begin{equation}\label{eq:ordinary-sr}
  \sr(n,x) \coloneq \sum_{\lambda\vdash n}\frac{1}{\spol(\lambda,x)}.
\end{equation}

\subsubsection{The numerator and denominator of $\sr(n,x)$}
A common denominator for \eqref{eq:ordinary-sr}
is
\[
   \den^*(n,x) \coloneq \prod_{i\ge1}(1+x^i)^{\floor{n/i}},
\]
and the corresponding numerator is
\[
   \num^*(n,x) \coloneq \sum_{\lambda\vdash n}h_\lambda(x),
\]
where
\[
   h_\lambda(x) \coloneq
   \prod_{i\ge1}(1+x^i)^{\floor{n/i}-m_\lambda(i)}.
\]
(The products are finite in effect,
since $\floor{n/i}=0$ for $i>n$.
The exponents of $h_\lambda(x)$ are nonnegative
because a partition of $n$ can contain at most $\floor{n/i}$ parts equal to $i$.)
In other words,
\begin{equation}\label{eq:unreduced-ordinary}
   \sr(n,x)=\frac{\num^*(n,x)}{\den^*(n,x)}.
\end{equation}

However the quotient \eqref{eq:unreduced-ordinary} need not be in lowest terms.
The first cancellation considered by \bbfm\ removes the common divisor of the
individual summands.
Define
\[
   G(n,x) \coloneq \gcd\{h_\lambda(x):\lambda\vdash n\},
\]
and then define
\begin{equation}\label{eq:reduced-defs}
   \num(n,x) \coloneq \frac{\num^*(n,x)}{G(n,x)}
   \qquad {\text {\rm and}}\qquad
   \den(n,x) \coloneq \frac{\den^*(n,x)}{G(n,x)}.
\end{equation}
Therefore, we have
\begin{equation}\label{refinement}
   \sr(n,x)=\frac{\num(n,x)}{\den(n,x)}.
\end{equation}
We use the empty-partition conventions
$G(0,x) = 1$, $h_{\emptyset}(x) = 1$, $\num(0,x) = 1$.

We stress that $G(n,x)$ is \emph{not} defined to be
$\gcd(\num^*(n,x),\den^*(n,x))$.
It is the common factor of the summands $h_\lambda(x)$.
Therefore, it is conceivable \emph{a priori} that
$\num(n,x)$ and $\den(n,x)$ are not relatively prime as polynomials in $\ZZ[x]$.

\subsubsection{Example for $n=4$}
Here we write out the example $n = 4$ in full.
The five partitions of $4$ are
\[
   (4),\qquad (3,1),\qquad (2,2),\qquad (2,1,1),\qquad (1,1,1,1),
\]
so their subsum polynomials are
\begin{align*}
  \lambda &= (4) &\implies&& \spol(\lambda, x) &= 1+x^4 \\
  \lambda &= (3,1) &\implies&& \spol(\lambda, x) &= (1+x^3)(1+x) \\
  \lambda &= (2,2) &\implies&& \spol(\lambda, x) &= (1+x^2)^2 \\
  \lambda &= (2,1,1) &\implies&& \spol(\lambda, x) &= (1+x^2)(1+x) \\
  \lambda &= (1,1,1,1) &\implies&& \spol(\lambda, x) &= (1+x)^4.
\end{align*}
Therefore, we have
\[
\begin{aligned}
\sr(4,x)
=\frac{1}{1+x^4}
  +\frac{1}{(1+x^3)(1+x)}
  +\frac{1}{(1+x^2)^2}
  +\frac{1}{(1+x^2)(1+x)^2}
  +\frac{1}{(1+x)^4}.
\end{aligned}
\]
The unreduced common denominator is
\[
   \den^*(4,x)
   =(1+x)^4(1+x^2)^2(1+x^3)(1+x^4).
\]
The five terms $h_\lambda(x)=\den^*(4,x)/\spol(\lambda,x)$ are
\begin{align*}
  \lambda &= (4) &\implies&&
    h_\lambda(x) &= (1+x)^4(1+x^2)^2(1+x^3) \\
  \lambda &= (3,1) &\implies&&
    h_\lambda(x) &= (1+x)^3(1+x^2)^2\phantom{(1+x^3)}(1+x^4) \\
  \lambda &= (2,2) &\implies&&
    h_\lambda(x) &= (1+x)^4\phantom{(1+x^2)^2}(1+x^3)(1+x^4) \\
  \lambda &= (2,1,1) &\implies&&
    h_\lambda(x) &= (1+x)^2(1+x^2)\phantom{^2}(1+x^3)(1+x^4) \\
  \lambda &= (1,1,1,1) &\implies&&
    h_\lambda(x) &= \phantom{(1+x)^4}(1+x^2)^2(1+x^3)(1+x^4)
\end{align*}
and $\num^*(4,x) = \sum_{h \vdash 4} h_\lambda(x)$.
Since $1+x$ divides $1+x^3$, each of these polynomials is divisible by $1+x$,
and in fact this is the greatest common factor:
\[
   G(4,x)=\gcd\left \{  h_\lambda(x) \ : \lambda\vdash 4\right \}=1+x.
\]
After dividing by this common factor, we obtain
\[
   \den(4,x)=\frac{\den^*(4,x)}{1+x}
   =(1+x)^3(1+x^2)^2(1+x^3)(1+x^4),
\]
and
\[
\begin{aligned}
\num(4,x)
&=\frac{\num^*(4,x)}{1+x}  \\
&=5x^{10}+8x^9+15x^8+14x^7+24x^6+20x^5
  +24x^4+14x^3+15x^2+8x+5.
\end{aligned}
\]
In this case $\sr(4,x)=\frac{\num(4,x)}{\den(4,x)}$ is already in lowest terms.

\subsubsection{Restricted classes of partitions}
We refer to the partitions of $n$ as \emph{ordinary partitions}.
We also use the following special classes:
$\OO(n)$ is the set of odd partitions of $n$, meaning
partitions all of whose parts are odd; $\BB(n)$ is the set of binary
partitions of $n$, meaning partitions all of whose parts are powers of $2$;
and $\TT(n)$ is the set of ternary partitions of $n$, meaning partitions all
of whose parts are powers of $3$.

The same notation will be used for the restricted classes of partitions.  If
$\mathcal C$ is one of $\OO,\BB,\TT$, let $A_{\mathcal C}$ be its set of
allowed parts:
\[
   A_\OO \coloneq \{1,3,5,\ldots\},\qquad
   A_\BB \coloneq \{1,2,4,8,\ldots\},\qquad
   A_\TT \coloneq \{1,3,9,27,\ldots\}.
\]
For $\lambda\in\mathcal C(n)$, define
\begin{align*}
  h_{\mathcal C,\lambda}(x) &\coloneq
  \prod_{i\in A_{\mathcal C}}(1+x^i)^{\floor{n/i}-m_\lambda(i)} \\
  G_{\mathcal C}(n,x) &\coloneq
   \gcd_{\lambda\in\mathcal C(n)} h_{\mathcal C,\lambda}(x),
  \end{align*}
and
\[
   \num_{\mathcal C}(n,x) \coloneq \frac{1}{G_{\mathcal C}(n,x)}
      \sum_{\lambda\in\mathcal C(n)}h_{\mathcal C,\lambda}(x)
   \qquad {\text {\rm and}}\qquad
   \den_{\mathcal C}(n,x) \coloneq
   \frac{\prod_{i\in A_{\mathcal C}}(1+x^i)^{\floor{n/i}}}
        {G_{\mathcal C}(n,x)}.
\]

\subsection{\bbfm's Ten Conjectures}
In the recent paper \cite{BBFM},
\bbfm\ make ten explicit conjectures about these reciprocals
of partition polynomials.  We list them here in the notation fixed above, so
that the statements proved later can be compared directly with the original
conjectures.

\subsubsection{Ordinary partitions}
The first two conjectures concern the ordinary numerator and
ask whether it is irreducible, and then whether \eqref{refinement} is already
in lowest terms.

\begin{conjecture}\label{conj:conj1}
For $n\ge1$, the polynomial $\num(n,x)$ is irreducible over $\ZZ$.
\end{conjecture}

\begin{conjecture}\label{conj:conj2}
For $n\ge1$, we have $\gcd(\num(n,x),\den(n,x))=1$.
\end{conjecture}

The following two conjectures concern coefficient-shape properties of the
ordinary numerator and denominator.
Recall that a finite sequence is \emph{unimodal} if it weakly increases up to some point
and then weakly decreases;
and a finite sequence $(a_i)$ of nonnegative numbers is
\emph{log-concave} if $a_i^2\ge a_{i-1}a_{i+1}$ for every internal index $i$.

\begin{conjecture}\label{conj:conj3}
Write $\num(n,x)=\num_0(n,x)+\num_1(n,x)$, where $\num_0$ contains the even
powers of $x$ and $\num_1$ contains the odd powers of $x$.  Then
$\num_0(n,x)$ is unimodal.
\end{conjecture}

\begin{conjecture}\label{conj:conj4}
The sequence of coefficients of $\den(n,x)$ is log-concave except for
$n=3,5,6,7$.
\end{conjecture}

\subsubsection{Binary partitions}
\label{sec:bin}
The next three conjectures concern binary partitions.

\begin{conjecture}\label{conj:conj5}
For the binary numerator and denominator, we have
\[
   \gcd(\num_\BB(n,x),\den_\BB(n,x))=1
   \qquad(n\ge1).
\]
\end{conjecture}

\begin{conjecture}\label{conj:conj6}
  For $n>5$, the binary numerator $\num_\BB(n,x)$ is unimodal.
\end{conjecture}

\begin{remark}
  \cref{conj:conj6} in \cite{BBFM} originally
  asserted that $\num_\BB(n,x)$ is log-concave for $n\geq 2.$
  This fails for $n=4$, where
  \[ \num_\BB(4,x) = 4+10x+18x^2+18x^3+20x^4+18x^5+18x^6+10x^7+4x^8.  \]
  The coefficient sequence is unimodal, but it is not log-concave, since $18^2<18\cdot20$. The authors of [2] have informed us that their intended
log-concavity conjecture excludes the small cases \(n=4,5\), and remains open for
 \(n>5\).
\end{remark}

\begin{conjecture}\label{conj:conj7}
For $n\ge2$ and every $2^s\le n$, we have
\[
   (1+x^{2^s})\nmid \num_\BB(n,x).
\]
\end{conjecture}

\begin{remark}
It is noted in \cite{BBFM} that \cref{conj:conj5} and \cref{conj:conj7} are equivalent.
\end{remark}

\subsubsection{Odd and ternary partitions}
The remaining conjectures\footnote{The formulation of \cref{conj:conj10}
  printed in \cite{BBFM} reads $t(n) = \Delta(t(3n-2)/2^{2n})$.
  If read literally as applying the finite-difference
  operator to the sequence $\frac{t(3n-2)}{2^{2n}}$,
  then the statement is false, since, e.g.,
  $\frac{t(4)}{2^4}-\frac{t(1)}{2^2} =\frac{5}{16}-\frac14 =\frac1{16} \neq t(1) = 1$.
  We believe that the statement given in \cref{conj:conj10}
  is the intended reading (note we have used $t(3n+1) = t(3n)$).}
treated here involve odd and ternary partitions.  We
write $v_p(M)$ for the exponent of the prime $p$ in a
positive integer $M$.

\begin{conjecture}\label{conj:conj8}
If $o(m)$ denotes the largest odd divisor of $m$, then
$\num_\OO(n,-1)=o(n!)$ for $n \ge 1$.
\end{conjecture}

\begin{conjecture}\label{conj:conj9}
Let $s(n) \coloneq \num_\TT(n,-1)$ for $n \ge 1$.
Then $s(1)=s(2)=1$, and
$s(3n)=s(3n+1)=s(3n+2)=3^{v_3((3n)!)}$ for $n \ge 1$.
\end{conjecture}

\begin{conjecture}\label{conj:conj10}
For $n \ge 0$, let $t(n) \coloneq \num_\TT(n,1)$,
with the convention $t(0) \coloneq 1$,
Then $t(3n)=t(3n+1)=t(3n+2)$ for $n \ge 0$, and
\[ t(n)=\frac{t(3n)-t(3n-2)}{2^{2n}} \]
for $n \ge 1$.
In particular, $t(1)=t(2)=1$.
\end{conjecture}

\subsection{Results}
In this note we prove six of the ten conjectures
using elementary number theory and facts about cyclotomic polynomials.
The precise results are as follows.

\begin{theorem}[\cref{conj:conj2}]\label{thm:intro-ordinary}
For every $n\ge1$, we have
\[
   \gcd(\num(n,x),\den(n,x))=1.
\]
\end{theorem}

\begin{theorem}[\cref{conj:conj5,conj:conj7}]\label{thm:intro-binary}
For every $n\ge2$ and every $2^s\le n$, we have
\[
   1+x^{2^s}\nmid \num_\BB(n,x).
\]
Consequently, for every $n\ge1$, we have
\[
   \gcd(\num_\BB(n,x),\den_\BB(n,x))=1.
\]
\end{theorem}

\begin{theorem}[\cref{conj:conj8}]\label{thm:intro-odd}
For every $n\ge1$, we have
\[
   \num_\OO(n,-1)=o(n!),
\]
where $o(M)$ denotes the largest odd divisor of $M$.
\end{theorem}

\begin{theorem}[\cref{conj:conj9}]\label{thm:intro-ternary-minus}
For every $n\ge1$, we have
\[
   \num_\TT(n,-1)=3^{v_3(n!)}.
\]
In particular, for every $n\ge1$, we have
\[
   \num_\TT(3n,-1)=\num_\TT(3n+1,-1)=\num_\TT(3n+2,-1)
   =3^{v_3((3n)!)}.
\]
\end{theorem}

\begin{theorem}[\cref{conj:conj10}]\label{thm:intro-ternary-one}
  Let $t(n) \coloneq \num_\TT(n,1)$ and put $t(0) \coloneq 1$.
  Then
  \begin{align*}
    t(3n)=t(3n+1) &= t(3n+2) \text{ for } n \ge 0, \\
    t(3n)-t(3n-2) &= 2^{2n}t(n) \text{ for } n \ge 1.
  \end{align*}
\end{theorem}

\smallskip
The results above resolve the conjectures of Ballantine--Beck--Feigon--Maurischat that concern
coprimality, divisibility, and special values. For clarity, we summarize below the status of all
ten conjectures in their paper. The remaining unresolved cases are the conjectures concerning
irreducibility and coefficient-shape phenomena; in the binary case, the log-concavity assertion is
false as stated.

We note that several of the reductions used below build directly on the
structural analysis in \cite{BBFM}.  In particular, the exponent calculation
for the common divisor, the reduction of evaluations at roots of unity to
a remainder problem, and the corresponding binary reduction were already
developed in \cite[Lemmas 3.1 and 3.14]{BBFM} and indicated in
\cite[Section 4]{BBFM}.  Our contribution in Sections 3 and 4 is to combine
these reductions with a same-ray noncancellation argument, which rules out
the remaining possible cyclotomic factors and proves the corresponding
coprimality and divisibility conjectures.
\smallskip

\begin{table}[H]
\centering
\renewcommand{\arraystretch}{1.25}
\begin{tabular}{|c|p{0.60\textwidth}|p{0.28\textwidth}|}
\hline
\textbf{\#} & \textbf{Topic} & \textbf{Status in this paper} \\
\hline
1 & Irreducibility of $\operatorname{num}(n,x)$ over $\mathbb{Z}$ & Open \\
\hline
2 & $\gcd(\operatorname{num}(n,x),\operatorname{den}(n,x))=1$ & Proved \\
\hline
3 & Unimodality of the even part $\operatorname{num}_0(n,x)$ & Open \\
\hline
4 & Log-concavity of $\operatorname{den}(n,x)$ for $n\not \in \{3,5,6,7\}$ & Open \\
\hline
5 & $\gcd(\operatorname{num}_\BB(n,x),\operatorname{den}_\BB(n,x))=1$ & Proved \\
\hline
6 & Unimodality and log-concavity of $\operatorname{num}_\BB(n,x)$ & Disproved (corrected \(n>5\) form open) \\
\hline
7 &  $(1+x^{2^i})\nmid \operatorname{num}_\BB(n,x)$ for $2^i\le n$ & Proved \\
\hline
8 & $\operatorname{num}_\OO(n,-1)=o(n!)$ for odd partitions & Proved \\
\hline
9 & Evaluation formula for $\operatorname{num}_\TT(n,-1)$  & Proved \\
\hline
10 & Recurrence formula for $\operatorname{num}_\TT(n,1)$ & Proved in corrected form \\
\hline
\end{tabular} \vskip.05in
\caption{Status of the BBFM conjectures addressed in this paper.}
\label{tab:bbfm-conjectures}
\end{table}

\medskip
The paper is organized as follows.
In \cref{sec:sec2} we record the
    elementary cyclotomic facts that will be used throughout the proofs. In \cref{sec:sec3}, we prove \cref{conj:conj2}
    for ordinary partitions by showing that the reduced numerator
    does not vanish at any root of unity that can occur as a zero of the reduced denominator.
    The key step is a reduction to a smaller remainder problem in which all surviving summands
    lie on a single open ray in the complex plane.
In \cref{sec:sec4}, we adapt the same idea to
    binary partitions, proving the nondivisibility statement of \cref{conj:conj7} and
    hence the equivalent coprimality statement of \cref{conj:conj5}.
In \cref{sec:sec5},
    we prove the evaluation in \cref{conj:conj8}
    for odd partitions by specializing at $x=-1$,
    where only the all-ones partition survives after cancellation.
In \cref{sec:sec6}, we treat ternary
    partitions: the specialization at $x=-1$ gives \cref{conj:conj9}, while the
    specialization at $x=1$ leads to a simple generating-function identity and
    proves the natural finite-difference form of \cref{conj:conj10}.
The appendix documents formal proofs of the theorems we proved,
and provides links to the relevant files for interested readers.
(Mathematicians not interested in automated theorem proving
can thus safely ignore the appendix.)

\section*{Acknowledgments}

We thank Cristina Ballantine, George Beck, Brooke Feigon, and Kathrin
Maurischat for introducing the subsum-polynomial framework and for formulating
the conjectures studied in this paper. We are especially grateful to Cristina Ballantine, Brooke
Feigon, and Kathrin Maurischat for a careful reading of an earlier draft. Their work provides
the starting point for the arguments here.

\section{Elementary cyclotomic facts}\label{sec:sec2}
Let $\Phi_m$ denote the $m$\textsuperscript{th} cyclotomic polynomial, irreducible over $\mathbb Z$.
We record two elementary facts about $\Phi_m$.
They are included to make the proof accessible to readers who have seen cyclotomic polynomials
but may not have used them in this way.
The first lemma identifies exactly which factors $1+x^i$ contain a given cyclotomic divisor.

\begin{lemma}\label{lem:cyclotomic-divisibility}
Let $d,i \ge 1$.
Then $\Phi_{2d}(x)$ divides $1+x^i$ if and only if $i=dj$ for some odd integer $j\ge1$.
When this occurs, $\Phi_{2d}(x)$ divides $1+x^i$ exactly once.
\end{lemma}

\begin{proof}
Let $\zeta_{2d} \coloneq e^{\pi\ii/d}$.  Then $\zeta_{2d}$ is a primitive
$2d$-th root of unity and $\zeta_{2d}^d=-1$.  The polynomial
$\Phi_{2d}$ divides $1+x^i$ exactly when $\zeta_{2d}^i=-1$.  This happens
exactly when
\[ i\equiv d \pmod{2d}, \]
which is the same as saying $i=dj$ for an odd integer $j\ge1$.
The second part follows by noting $1+x^i$ has no double roots.
\end{proof}

The second lemma records the values at $1$ and $-1$ that will be needed in
the odd and ternary specializations.

\begin{lemma}\label{lem:cyclotomic-values}
For $m>1$, we have
\[
   \Phi_m(1)=
   \begin{cases}
      p, & m=p^a\text{ is a power of a prime }p,\\
      1, & m\text{ is not a prime power}.
   \end{cases}
\]
If $m>1$ is odd, then $\Phi_{2m}(-1)=\Phi_m(1)$.  In particular, for
$a\ge1$, we have
\[
   \Phi_{2\cdot3^a}(1)=1,
   \qquad
   \Phi_{2\cdot3^a}(-1)=3.
\]
\end{lemma}

\begin{proof}
Let $\mu$ denote the classical M\"obius function.  The standard product
formula
\[
   \Phi_m(x)=\prod_{e\mid m}(x^e-1)^{\mu(m/e)}
\]
implies, after taking the limit as $x\to1$, that
\[
   \Phi_m(1)=\prod_{e\mid m}e^{\mu(m/e)}.
\]
For a fixed prime $p$, let $v_p$ denote the usual $p$-adic valuation.
The exponent of $p$ in this product is
\[
   \sum_{e\mid m}\mu(m/e)v_p(e).
\]
Writing $e=m/f$, this becomes
\[
   \sum_{f\mid m}\mu(f)(v_p(m)-v_p(f)).
\]
Since $m>1$, $\sum_{f\mid m}\mu(f)=0$.  Therefore, the exponent is
$-\sum_{f\mid m}\mu(f)v_p(f)$.  This is $1$ if $m$ is a power of $p$, and
is $0$ otherwise.  This proves the formula for $\Phi_m(1)$.

If $m$ is odd, then $\Phi_{2m}(x)=\Phi_m(-x)$.  Therefore
$\Phi_{2m}(-1)=\Phi_m(1)$.  Finally,
$\Phi_{2\cdot3^a}(x)=\Phi_{3^a}(-x)$, so
$\Phi_{2\cdot3^a}(-1)=\Phi_{3^a}(1)=3$, while
\[ \Phi_{2\cdot3^a}(1)=\Phi_{3^a}(-1)=1. \qedhere \]
\end{proof}

\section{Ordinary partitions: proof of \cref{conj:conj2}}\label{sec:sec3}
Throughout this section,
$m_\lambda(i)$, $h_\lambda$, and $G(n,x)$
have the meanings fixed in the introduction.
We do not need a closed formula for the reduced denominator.
In what follows it will be useful to write
\[
   \num(n,x)=\sum_{\lambda\vdash n}q_\lambda(x),
\]
where
\[
   q_\lambda(x) \coloneq \frac{h_\lambda(x)}{G(n,x)}.
\]

We now prove \cref{thm:intro-ordinary}.
Since $\den(n,x)=\den^*(n,x)/G(n,x)$ divides $\den^*(n,x)$, every irreducible
factor of $\den(n,x)$ is a cyclotomic factor of some $1+x^i$ with
$1\le i\le n$.  Therefore, every irreducible factor of $\den(n,x)$ has
the form $\Phi_{2d}(x)$ for some $1\le d\le n$.  It is enough to prove
that $\num(n,x)$ does not vanish at a primitive $2d$-th root of unity for
any such $d$.  The first lemma computes how much of each possible
cyclotomic factor is removed by the common divisor $G(n,x)$. The following lemma is a cyclotomic reformulation of the exponent
calculation in \cite[Lemma 3.1]{BBFM}; the corresponding characterization
of divisibility by \(\Phi_{2d}\) is implicit in \cite[Lemma 3.14(a,b)]{BBFM}.
We include the short proof to keep the notation of the present paper
self-contained.

\begin{lemma}\label{lem:g-exponent}
Fix $n, d \ge 1$, and put $a \coloneq \floor{n/d}$.  The exponent of $\Phi_{2d}(x)$ in
$G(n,x)$ is
\[
   \sum_{\substack{j>1\\ j\text{ odd}}}\floor{\frac{n}{dj}}.
\]
Consequently the exponent of $\Phi_{2d}(x)$ in
$q_\lambda(x)=h_\lambda(x)/G(n,x)$ is
\[
   a-\sum_{\substack{j\ge1\\ j\text{ odd}}}m_\lambda(dj).
\]
\end{lemma}

\begin{proof} This is the same minimization argument as in \cite[Lemma 3.1]{BBFM}, with
\((1+x^d)\)-divisibility translated to \(\Phi_{2d}\)-divisibility using
Lemma~1.
By \cref{lem:cyclotomic-divisibility}, the exponent of $\Phi_{2d}(x)$ in
$h_\lambda(x)$ is
\[
   \sum_{\substack{j\ge1\\ j\text{ odd}}}
   \left(\floor{\frac{n}{dj}}-m_\lambda(dj)\right).
\]
The exponent in $G(n,x)$ is the minimum of this quantity over all partitions
$\lambda\vdash n$.

The multiplicity sum
$\sum_{j\ge1,\,j\text{ odd}}m_\lambda(dj)$ is at most $\floor{n/d}$, because
the counted parts have total size at least
$d\sum_{j\ge1,\,j\text{ odd}}m_\lambda(dj)$.  This upper bound is attained by a
partition containing exactly $a=\floor{n/d}$ parts equal to $d$, with the
remaining $n-ad$ filled by ones.  Hence the minimum exponent is obtained by
subtracting $a$ from the full floor sum:
\[
   \sum_{\substack{j\ge1\\j\text{ odd}}}\floor{\frac{n}{dj}}-a
   =\sum_{\substack{j>1\\j\text{ odd}}}\floor{\frac{n}{dj}}.
\]
Subtracting this minimum from the exponent in $h_\lambda$ gives the
exponent in $q_\lambda$.
\end{proof}

The next lemma identifies the only summands that can survive after evaluating
at a primitive $2d$-th root of unity.  It reduces the problem from $n$ to
the remainder of $n$ modulo $d$. This remainder reduction is essentially \cite[Lemma 3.14(c,d)]{BBFM}.
The constant appearing below is denoted there in terms of \(f(n,d,x)\), and
an explicit product formula for it is given in \cite[Proposition 3.17]{BBFM}.
For our purposes, only its nonvanishing is needed.

\begin{lemma}[Remainder reduction]\label{lem:remainder-reduction}
Let $n, d \ge 1$, and write
\[
   n=ad+r,
   \qquad
   a \coloneq \floor{n/d},
   \qquad
   0\le r<d.
\]
Then we have
\[
   \num(n,\zeta_{2d})=C_{n,d}\,\num(r,\zeta_{2d})
\]
for some nonzero complex number $C_{n,d}$.  In particular, we have
\[
   \num(n,\zeta_{2d})\ne0
   \quad\Longleftrightarrow\quad
   \num(r,\zeta_{2d})\ne0.
\]
\end{lemma}

\begin{proof}
We first start by analyzing which terms of the sum
\[ \num(n, \zeta_{2d}) = \sum_{\lambda \vdash n} q_\lambda(\zeta_{2d}) \]
are nonzero.
By \cref{lem:g-exponent}, the exponent of $\Phi_{2d}$ in the reduced
summand $q_\lambda$ is
\[
   a-\sum_{\substack{j\ge1\\j\text{ odd}}}m_\lambda(dj).
\]
At $x=\zeta_{2d}$, a summand survives if and only if this exponent is zero, i.e.,
\[ \sum_{\substack{j\ge1\\j\text{ odd}}}m_\lambda(dj)=a.  \]
If any counted part $dj$ has $j>1$, then one of the $a$ counted parts has size at least $3d$,
so the counted parts alone have total size at least
$(a-1)d+3d=(a+2)d>ad+r=n$.  This is impossible.
We conclude the surviving partitions are exactly
\[
   \lambda=(d^a,\nu),
   \qquad
   \nu\vdash r,
\]
with the empty partition allowed when $r=0$.
In other words,
\begin{equation}
  \num(n, \zeta_{2d}) = \sum_{\nu \vdash r} q_{(d^a, \nu)}(\zeta_{2d}).
  \label{eq:nu_sum}
\end{equation}

Next we rewrite $q_{(d^a, \nu)}$ in terms of $q_\nu$ for $\nu \vdash r$.
Comparison of exponents gives the identity, in the field
of rational functions,
\[
   q_{(d^a,\nu)}(x)=F_{n,d}(x)q_\nu(x),
\]
where
\[
   F_{n,d}(x) \coloneq
   \frac{G(r,x)}{G(n,x)}
   \prod_{i\ne d}(1+x^i)^{\floor{n/i}-\floor{r/i}}.
\]
This factor is independent of $\nu$, so \eqref{eq:nu_sum} implies
\[
  \num(n,\zeta_{2d})=C_{n,d}\,\num(r,\zeta_{2d}),
  \] 
  where $C_{n,d} \coloneq F_{n,d}(\zeta_{2d}).$ In the notation of \cite[Lemma 3.14]{BBFM}, this constant is
\(f(n,d,\zeta_{2d})\); see also \cite[Proposition 3.17]{BBFM} for an
explicit formula.
It remains only to verify that $C_{n,d} \neq 0$.

A factor $1+x^i$ can vanish at
$\zeta_{2d}$ only through the cyclotomic factor $\Phi_{2d}$.  Since
$r<d$, no factor occurring in $\den^*(r,x)$, and hence no factor of
$G(r,x)$, is divisible by $\Phi_{2d}$.  Thus $G(r,\zeta_{2d})\ne0$.
The product
\[
   \prod_{i\ne d}(1+x^i)^{\floor{n/i}-\floor{r/i}}
\]
has $\Phi_{2d}$-exponent
\[
   \sum_{\substack{j>1\\j\text{ odd}}}
      \left(\floor{\frac{n}{dj}}-\floor{\frac{r}{dj}}\right)
   =\sum_{\substack{j>1\\j\text{ odd}}}\floor{\frac{n}{dj}},
\]
because $r<d$.
By \cref{lem:g-exponent}, this is exactly the $\Phi_{2d}$-exponent of $G(n,x)$.
After cancelling these equal powers of $\Phi_{2d}$,
every remaining numerator and denominator factor is nonzero at $\zeta_{2d}$.
\end{proof}

The reduction leaves the case $0\le r<d$.  The following same-ray argument
shows that the remaining terms cannot cancel.

\begin{lemma}\label{lem:same-ray}
Let $d\ge1$ and $0\le r<d$.  Then
\[
   \num(r,\zeta_{2d})\ne0.
\]
\end{lemma}

\begin{proof}
If $r=0$, then $\num(0,x)=1$.  Assume $1\le r<d$.  For $1\le i\le r$,
\[
   1+\zeta_{2d}^{\,i}
   =1+e^{\pi\ii i/d}
   =2\cos\left(\frac{\pi i}{2d}\right)e^{\pi\ii i/(2d)}.
\]
The cosine is strictly positive because $0<i<d$.  Thus each factor
$1+\zeta_{2d}^{\,i}$ has positive length and argument $\pi i/(2d)$.

For a partition $\nu\vdash r$, the value
$h_\nu(\zeta_{2d})$ has argument
\[
   \frac{\pi}{2d}
   \sum_{i=1}^r i\bigl(\floor{r/i}-m_\nu(i)\bigr).
\]
But $\sum_i i m_\nu(i)=r$, so this argument equals
\[
   \frac{\pi}{2d}\left(\sum_{i=1}^r i\floor{r/i}-r\right),
\]
which is independent of $\nu$.  Also, every factor involved is nonzero.
Therefore all numbers $h_\nu(\zeta_{2d})$ lie on a single open ray
from the origin.

Since $r<d$, no factor $1+x^i$ with $1\le i\le r$ is divisible by
$\Phi_{2d}$.  Thus $G(r,\zeta_{2d})\ne0$.  Dividing by this fixed nonzero
number only rotates and rescales the ray.  Hence the reduced summands
$q_\nu(\zeta_{2d})$ also lie on a single open ray.  Their sum cannot
be zero.
\end{proof}

\begin{proof}[Proof of \cref{thm:intro-ordinary}]
Every irreducible factor of $\den(n,x)$ is some $\Phi_{2d}(x)$ with
$1\le d\le n$.  Fix such a $d$.  By \cref{lem:remainder-reduction} and \cref{lem:same-ray},
\[
   \num(n,\zeta_{2d})\ne0.
\]
Hence $\Phi_{2d}\nmid\num(n,x)$.  Therefore no nonconstant irreducible factor
of $\den(n,x)$ divides $\num(n,x)$.  Since $\den(n,x)$ has constant term
$1$, no rational integer prime divides all its coefficients.  Thus
$\gcd(\num(n,x),\den(n,x))=1$ in $\ZZ[x]$.
\end{proof}

\section{Binary partitions: proof of \cref{conj:conj5} and \cref{conj:conj7}}\label{sec:sec4}

We next adapt the same root-of-unity idea to binary partitions.  Here the
allowed parts are $1,2,4,\ldots$, and the factors $1+x^{2^k}$ are already
pairwise coprime cyclotomic polynomials.  By the definitions in the
introduction,
\[
   h_{\BB,\lambda}(x) \coloneq
   \prod_{k\ge0}\left( 1+x^{2^k} \right)^{\floor{n/2^k}-m_\lambda(2^k)}.
\]

The following observation is \cite[Remark 6]{BBFM}; we recall the brief
argument for completeness.
\begin{lemma}
  We always have $G_\BB(n,x) = 1$, for every $n$.
  Hence,
  \[
     \num_\BB(n,x)=\sum_{\lambda\in\BB(n)}h_{\BB,\lambda}(x)
     \qquad {\text {\rm and}}\qquad
     \den_\BB(n,x)=\prod_{k\ge0}(1+x^{2^k})^{\floor{n/2^k}}.
  \]
\end{lemma}
\begin{proof}
  Since $1+x^{2^k}=\Phi_{2^{k+1}}(x)$, these factors are pairwise coprime
  irreducibles.  For a fixed $k$, put $D \coloneq 2^k$.  The binary partition with
  $\floor{n/D}$ parts equal to $D$, together with any binary partition of
  the remainder $n-D\floor{n/D}$, makes the exponent of $1+x^D$ equal to zero.
  Thus, no $1+x^{2^k}$ divides all binary summands.
\end{proof}
The same pairwise coprimality shows that \cref{conj:conj5} is
equivalent to \cref{conj:conj7}: the binary numerator is coprime to
the binary denominator exactly when no factor $1+x^{2^s}$ divides
$\num_\BB(n,x)$.

We use the binary remainder reduction indicated in
\cite[Section 4, final paragraph]{BBFM}, which is the binary analogue of
\cite[Lemma 3.14]{BBFM}.  The additional point below is the same-ray
noncancellation argument for the remaining sum.

\begin{proof}[Proof of \cref{thm:intro-binary}]
Fix $D \coloneq 2^s\le n$, and let
\[
   \zeta \coloneq e^{\pi\ii/D}.
\]
Then $1+x^D$ is the minimal polynomial of $\zeta$, since $D$ is a power of
$2$.  It is enough to show that $\num_\BB(n,\zeta)\ne0$.

Write $n=aD+r$, with $a=\floor{n/D}$ and $0\le r<D$.  In a binary
summand $h_{\BB,\lambda}(x)$, the only factor $1+x^{2^k}$ that vanishes
at $x=\zeta$ is $1+x^D$.  Hence a summand can be nonzero only if
\[
   m_\lambda(D)=a.
\]
The surviving partitions are exactly
\[
   \lambda=(D^a,\rho),
   \qquad
   \rho\in\BB(r).
\]
For such $\lambda$, comparison of the exponents gives the exact identity
\[
   h_{\BB,(D^a,\rho)}(x)=C_{n,D}(x)h_{\BB,\rho}(x),
\]
where
\[
   C_{n,D}(x) \coloneq
   \prod_{\substack{k\ge0\\2^k\ne D}}
   (1+x^{2^k})^{\floor{n/2^k}-\floor{r/2^k}}.
\]
This product is independent of $\rho$.  Since $1+\zeta^{2^k}\ne0$ for
$2^k\ne D$, we have $C_{n,D}(\zeta)\ne0$.  Therefore, we have
\[
   \num_\BB(n,\zeta)=C_{n,D}(\zeta)\num_\BB(r,\zeta).
\]
This is the binary analogue of the reduction used in
\cite[Lemma 3.14(c,d)]{BBFM} and noted after \cite[Theorem 4.10]{BBFM}.

It remains to prove $\num_\BB(r,\zeta)\ne0$.  If $r=0$, this is immediate.
Assume $1\le r<D$.  For every binary part $2^k\le r$,
\[
   1+\zeta^{2^k}
   =2\cos\left(\frac{\pi2^k}{2D}\right)e^{\pi\ii2^k/(2D)},
\]
and the cosine is positive.  Thus each such factor has positive length and
argument $\pi2^k/(2D)$.  For $\rho\in\BB(r)$, the argument of
$h_{\BB,\rho}(\zeta)$ is
\[
   \frac{\pi}{2D}
   \sum_{2^k\le r} 2^k \cdot \left(\floor{r/2^k}-m_\rho(2^k)\right).
\]
Since $\sum_k2^k m_\rho(2^k)=r$, this argument is independent of $\rho$.  All
surviving terms are nonzero and lie on a single open ray.  Their sum is nonzero.
Therefore $\num_\BB(n,\zeta)\ne0$, so $1+x^D\nmid\num_\BB(n,x)$.
This proves \cref{conj:conj7}, and the coprimality statement follows from the equivalence above.
\end{proof}

\section{Odd partitions: proof of \cref{conj:conj8}}\label{sec:sec5}
We now specialize the restricted numerator for odd partitions at $x=-1$.
The definitions from the introduction give
\[
   h_{\OO,\lambda}(x)
    \coloneq \prod_{\substack{i\ge1\\ i\text{ odd}}}
      (1+x^i)^{\floor{n/i}-m_\lambda(i)},
\]
with common divisor
\[
   G_\OO(n,x) \coloneq \gcd_{\lambda\in\OO(n)}h_{\OO,\lambda}(x)
\]
and reduced numerator
\[
   \num_\OO(n,x) \coloneq \frac{1}{G_\OO(n,x)}
   \sum_{\lambda\in\OO(n)}h_{\OO,\lambda}(x)
   = \sum_{\lambda\in\OO(n)} q_{\OO,\lambda}(x)
\]
where the reduced summands $q_{\OO,\lambda}(x) \coloneq \frac{h_{\OO,\lambda}(x)}{G_{\OO}(n,x)}$
are defined in the obvious way.
The proof shows that, after this cancellation, only the all-ones partition
survives at $x=-1$.

\begin{proof}[Proof of \cref{thm:intro-odd}]
For an odd positive integer $e$, the factor $\Phi_{2e}$ divides $1+x^i$,
with $i$ odd, exactly when $i=ej$ for an odd integer $j$.  Therefore the
exponent of $\Phi_{2e}$ in $h_{\OO,\lambda}$ is
\[
   \sum_{\substack{j\ge1\\j\text{ odd}}}
      \left(\floor{\frac{n}{ej}}-m_\lambda(ej)\right).
\]
As before, this is minimized by maximizing the number of parts of the form $ej$ with $j$ odd.
In symbols, we have $\sum_{\substack{j\ge1\text{ odd}}} m_\lambda(ej) \le \floor{n/e}$
with equality when $\lambda$ has $\floor{n/e}$ parts equal to $e$.
Thus the exponent of $\Phi_{2e}$ in $G_\OO(n,x)$ is
\[
   \sum_{\substack{j>1\\j\text{ odd}}}\floor{\frac{n}{ej}}.
\]
Consequently, in $q_{\OO, \lambda}$, the exponent of $\Phi_{2e}$ is
\[
   \floor{\frac{n}{e}}-
   \sum_{\substack{j\ge1\\j\text{ odd}}}m_\lambda(ej).
\]

We now specialize to $x = -1$.
The only cyclotomic factor that vanishes is $\Phi_2(x)=x+1$, which corresponds to $e=1$.
The exponent of $\Phi_2$ in the reduced summand
is
\[ n-\ell(\lambda), \]
where $\ell(\lambda)$ is the number of parts of $\lambda$.
Among odd partitions of $n$, the maximum possible length is $n$,
attained uniquely by
$(1^n)$.  Hence all reduced summands vanish at $-1$ except the one
corresponding to $(1^n)$, that is to say,
\[ \num_\OO(n, -1) = q_{\OO, (1^n)}(-1)
  = \prod_{\substack{e>1\\e\text{ odd}}} \Phi_{2e}(-1)^{\floor{n/e}}. \]
By \cref{lem:cyclotomic-values}, $\Phi_{2e}(-1)=p$ if
$e=p^a$ is a power of an odd prime $p$, and $\Phi_{2e}(-1)=1$ otherwise.
Therefore
\[
   \num_\OO(n,-1)=
   \prod_{\substack{p\text{ odd prime}}}\prod_{a\ge1}
      p^{\floor{n/p^a}}=o(n!),
\]
by Legendre's formula with the prime $2$ omitted.
\end{proof}

\section{Ternary partitions: proofs of \cref{conj:conj9} and \cref{conj:conj10}}\label{sec:sec6}
Finally, we treat ternary partitions.  The same cancellation analysis gives
both the evaluation at $x=-1$ and the generating function needed at $x=1$.
We use
\[
   h_{\TT,\lambda}(x)
    \coloneq \prod_{k\ge0}(1+x^{3^k})^{\floor{n/3^k}-m_\lambda(3^k)},
\]
\[
   G_\TT(n,x) \coloneq \gcd_{\lambda\in\TT(n)}h_{\TT,\lambda}(x),
   \qquad
   \num_\TT(n,x)
   \coloneq \frac{1}{G_\TT(n,x)} \sum_{\lambda\in\TT(n)}h_{\TT,\lambda}(x)
   = \sum_{\lambda \in \TT(n)} q_{\TT, \lambda}(x).
\]
For every $k\ge0$, we have
\[
   1+x^{3^k}=\prod_{a=0}^k\Phi_{2\cdot3^a}(x).
\]
Therefore, the exponent of $\Phi_{2\cdot3^a}$ in $h_{\TT,\lambda}$ is
\[
   \sum_{k\ge a}\left(\floor{\frac{n}{3^k}}-m_\lambda(3^k)\right).
\]
As in the previous sections, this is minimized by using
$\floor{n/3^a}$ parts equal to $3^a$ and filling the remainder with ones.
Hence, the exponent of $\Phi_{2\cdot3^a}$ in $G_\TT(n,x)$ is
\[
   \sum_{k>a}\floor{\frac{n}{3^k}},
\]
and the exponent in a reduced summand $q_{\TT,\lambda}$ is
\[
   \floor{\frac{n}{3^a}}-
   \sum_{k\ge a}m_\lambda(3^k).
\]

\begin{proof}[Proof of \cref{thm:intro-ternary-minus}]
At $x=-1$, the only vanishing factor is $\Phi_2=x+1$, corresponding to $a=0$.
Its exponent in the reduced summand $q_{\TT,\lambda}$ is
\[ n-\ell(\lambda).  \]
The unique ternary partition of $n$ with $n$ parts is $(1^n)$.  Therefore
only the reduced summand $q_{\TT,(1^n)}$ survives at $-1$.

For this summand, the reduced exponent of $\Phi_{2\cdot3^a}$ is $0$ for
$a=0$, and is $\floor{n/3^a}$ for $a\ge1$.  Therefore, we have
\[
  \num_\TT(n,-1) = q_{\TT,(1^n)}(-1)
  = \prod_{a\ge1}\Phi_{2\cdot3^a}(-1)^{\floor{n/3^a}}.
\]
By \cref{lem:cyclotomic-values}, $\Phi_{2\cdot3^a}(-1)=3$.  Therefore, we have
\[ \num_\TT(n,-1) = 3^{\sum_{a\ge1}\floor{n/3^a}} = 3^{v_3(n!)}.  \]
In view of
\[ v_3((3n)!) = v_3((3n+1)!) = v_3((3n+2)!), \]
the block form in \cref{conj:conj9} follows.
\end{proof}

\begin{proof}[Proof of \cref{thm:intro-ternary-one}]
At $x=1$, \cref{lem:cyclotomic-values} gives
$\Phi_{2\cdot3^a}(1)=1$ for $a\ge1$.  Thus the value of $G_\TT(n,x)$ at
$x=1$ is determined only by the factor $\Phi_2(1)=2$.  Its exponent in
$G_\TT(n,x)$ is
\[
   \sum_{k\ge1}\floor{\frac{n}{3^k}}.
\]
For a ternary partition $\lambda\in\TT(n)$, we have
\[
   h_{\TT,\lambda}(1)=
   2^{\sum_{k\ge0}\floor{n/3^k}-\ell(\lambda)}.
\]
After division by $G_\TT(n,1) = 2^{\sum_{k \ge 1} \floor{n/3^k}}$, the reduced summand has value
\[
  q_{\TT, \lambda}(1) = 2^{n-\ell(\lambda)}.
\]
Therefore, we obtain
\[
   t(n)=\num_\TT(n,1)=\sum_{\lambda\in\TT(n)}2^{n-\ell(\lambda)}.
\]

This has the ordinary generating function
\[
   T(q) \coloneq \sum_{n\ge0}t(n)q^n
   =\prod_{j\ge0} \left( 1-2^{3^j-1}q^{3^j} \right)^{-1}.
\]
The factor with $j=0$ is $(1-q)^{-1}$.  Reindexing the remaining product,
we find
\[
  \prod_{j\ge1} \left( 1-2^{3^j-1}q^{3^j} \right)^{-1}
  = \prod_{j\ge0} \left( 1-2^{3^j-1}(4q^3)^{3^j} \right)^{-1}
  =T(4q^3).
\]
Thus, we find that
\[
   T(q)=\frac{1}{1-q}T(4q^3),
   \qquad\text{or}\qquad
   (1-q)T(q)=T(4q^3).
\]
Comparing coefficients of $q^m$, with the convention $t(-1)=0$, gives
\[
   t(m)-t(m-1)=0
\]
unless $m$ is divisible by $3$.  This proves
\[
   t(3n)=t(3n+1)=t(3n+2)\qquad(n\ge0).
\]
Comparing coefficients of $q^{3n}$ gives
\[
   t(3n)-t(3n-1)=4^n t(n).
\]
By block constancy, $t(3n-1)=t(3n-2)$.  Hence
\[
   t(3n)-t(3n-2)=4^n t(n)=2^{2n}t(n),
\]
which is the claimed finite-difference formula.
\end{proof}

\section{Appendix: Autonomous proof production and formalization}\label{sec:AxiomProver}
At Axiom Math, we are developing AxiomProver,
an AI system for mathematical research based on autoformalization.
As a test case, we gave AxiomProver
ten tasks: to autonomously formalize and prove Conjectures~\ref{conj:conj1}-\ref{conj:conj10},
offering examples of AI assistance as well as limitations in mathematical research.
This appendix is separate from the rest of the manuscript.

\subsection{Division of labor}
The prose exposition of this paper, including this appendix, was written without the use of AI.
AxiomProver runs Lean,
an interactive theorem prover and a functional programming language built on dependent type theory,
designed to provide a rigorous computational framework for validating mathematical proofs \cite{Lean}.
This appendix is included not only to record that formal proofs were produced,
but also to clarify the division of labor between the human-written
mathematics and the AI-generated formalizations.  The
mathematical statements supplied to AxiomProver were written in natural
language and can be found in \cite{BBFM}.

The role of AxiomProver was to convert the non-formal statements
into Lean and to construct machine-checkable formal proofs. Thus, in the tasks
described below, AxiomProver was neither being asked to make any of the conjectures,
nor to decide the final organization of the paper; it was being asked to
produce complete formal proofs of the supplied statements.

Autoformalization involves automatically and autonomously converting
natural-language mathematics into machine-verifiable formal language.
Lean files are created to pass type checkers,
while the natural language papers aim to communicate ideas to readers.

\subsection{Task specification and results}
Given problems in natural or formal language,
AxiomProver attempts to generate a complete formal proof.
When it succeeds, the system outputs two files:
\begin{itemize}
  \item \texttt{problem.lean}, which formalizes the problem statement;
  \item \texttt{solution.lean}, which represents a complete proof in a formal language.
\end{itemize}
The ten tasks share a uniform layout --- for $i \in \{1,\dots,10\}$,
we gave the following files as input for the $i$-th task:
\begin{itemize}
    \item \texttt{Conj\_$i$\_Statement.tex}
    \item \texttt{task.md} contains a single line
    \begin{quote}
        Please read `\texttt{Conj\_$i$\_Statement.tex}' and formalize
        Conjecture~\texttt{\textbackslash{}ref\{conj:conj\_$i$\}}.
    \end{quote}
    \item \texttt{.environment} specifies the Lean version used is 4.28.0.
\end{itemize}

Out of the $10$ tasks, AxiomProver was able to autonomously formalize and prove
Conjecture $j$ for $j\in\{2,5,7,8,9,10\}$.
It also found the counterexample $n=4$ to the part of \cite[Conjecture 6]{BBFM}
asserting that $\num_\BB(n,x)$ has log-concave coefficients
(see the remark at the end of \cref{sec:bin}).
Here ``autonomously'' means that, after receiving the
specified input files, AxiomProver generated the corresponding
\texttt{problem.lean} and \texttt{solution.lean} files without further
mathematical intervention in the proof script. The authors did not hand-write
or repair the Lean proofs line by line. The resulting formal proofs were then
used as machine-checkable certificates for the proved conjectures.

\subsection{Artifacts}
The relevant files are posted in the following repository:
\begin{center}
\url{https://github.com/AxiomMath/PartitionPolynomial}
\end{center}
The version of Lean used is 4.28.0.
Compatibility with other versions is not guaranteed due to the evolving
nature of the Lean~4 compiler and its core libraries.

At first glance, the proofs generated by AxiomProver do not resemble the narrative outlined in this paper.
Converting a Lean file into a proof understandable by\textbf{} humans is challenging
because Lean is designed as code for a type-checker, not as a reader-friendly explanation.
It makes all the ``obvious'' bookkeeping explicit, such as rewrite steps, coercions, side conditions,
and case splits, and tends to follow the most convenient lemmas and tactics for the library,
rather than the most clear conceptual route.
A mathematician can usually condense this significantly by relying on shared historical context,
standard arguments, and informal identifications that Lean cannot assume.
As a result, writing a paper from Lean files is not just about reformatting.
The authors must understand the formal script, reconstruct the underlying ideas,
and then translate the code into a narrative that emphasizes the key insights
while safely omitting routine details.

\end{document}